\documentclass[12pt]{amsart}
\usepackage{a4wide}
\usepackage{amsmath,amsthm,amssymb}
\usepackage{graphicx}
\usepackage{hyperref}
\usepackage{tikz}
\usepackage{enumitem}
\tikzset{every node/.style={font=\scriptsize}}
\def\multiset#1#2{{\big(\kern-.2em \big(\genfrac{}{}{0pt}{}{#1}{#2} \big)\kern-.2em\big)}}
\def\Multiset#1#2{{\bigg(\kern-.4em\bigg(\genfrac{}{}{0pt}{}{#1}{#2}\bigg)\kern-.4em\bigg)}}
 
\theoremstyle{definition}

\newtheorem{theorem}{Theorem}[section]
\newtheorem{remark}[theorem]{Remark}
\newtheorem{corollary}[theorem]{Corollary}
\newtheorem{example}[theorem]{Example}
\newtheorem{proposition}[theorem]{Proposition}
\newtheorem{lemma}[theorem]{Lemma}

\newcommand\ind{\operatorname{index}}
\newcommand\outdeg{\operatorname{outdeg}}
\newcommand\vol{\operatorname{vol}}
\newcommand\PS{\operatorname{PS}}
\newcommand\Dyck{\operatorname{Dyck}}
\newcommand\LD{\Dyck}
\newcommand\DLD{\Dyck^{(2)}}

\newcommand\EW{\operatorname{EW}}

\newcommand\Car{\operatorname{Car}}

\newcommand\CT{\operatorname{CT}}
\newcommand\FF{\mathcal{F}}
\newcommand\GG{\widehat{G}}
\newcommand\ZZ{\mathbb{Z}}
\renewcommand\vec[1]{\mathbf{#1}}

\newcommand\Hyper[5]{
  {}_{#1}F_{#2} \left( 
    \begin{matrix}
      #3\\
      #4\\
    \end{matrix}
    ; #5
    \right)
}

\title{Volumes of flow polytopes related to caracol graphs}
\author{Jihyeug Jang and Jang Soo Kim}
\thanks{The authors were supported by NRF grants \#2019R1F1A1059081 and \#2016R1A5A1008055.} 
\address{Department of Mathematics,
Sungkyunkwan University (SKKU), Suwon, Gyeonggi-do 16419, South Korea}
\email{ab4242@skku.edu}
\email{jangsookim@skku.edu}

\begin{document}

\begin{abstract}
  Recently, Benedetti et al.~introduced an Ehrhart-like polynomial associated to
  a graph. This polynomial is defined as the volume of a certain flow polytope
  related to a graph and has the property that the leading coefficient is the
  volume of the flow polytope of the original graph with net flow vector
  $(1,1,\dots,1)$. Benedetti et al.~conjectured a formula for the Ehrhart-like
  polynomial of what they call a caracol graph. In this paper their conjecture
  is proved using constant term identities, labeled Dyck paths, and a cyclic
  lemma.
\end{abstract}

\maketitle

\section{Introduction}

The main objects in this paper are flow polytopes, which are certain polytopes associated to acyclic directed graphs
with net flow vectors. Flow polytopes have interesting connections with representation theory, geometry, analysis, and
combinatorics. A well known flow polytope is the Chan--Robbins--Yuen polytope, which is the flow polytope
of the complete graph $K_{n+1}$ with net flow vector $(1,0,\dots,0,-1)$. Chan, Robbins, and Yuen \cite{CRY2000}
conjectured that the volume of this polytope is a product of Catalan numbers. Their conjecture was proved by
Zeilberger \cite{Zeilberger1999} using the Morris constant term identity \cite{Morris1982}, which is equivalent to the
famous Selberg integral \cite{Selberg1944}.

Since the discovery of the Chan--Robbins--Yuen polytope, researchers have found
many flow polytopes whose volumes have nice product formulas, see
\cite{Benedetti2019, CKM2017,Meszaros2015, MMR2017,MMS2019, MSW2019, Yip2019} and
references therein. In this paper we add another flow polytope to this list by
proving a product formula for the volume of flow polytope coming from a caracol
graph, which was recently conjectured by Benedetti et al.~\cite{Benedetti2019}.
In order to state our results we introduce necessary definitions.

We denote $[n]:=\{1,2,\dots,n\}$. Throughout this paper, we only consider
connected directed graphs in which every vertex is an integer and every directed
edge is of the form $(i,j)$ with $i<j$.

Let $G$ be a directed graph on vertex set $[n+1]$ with $m$ directed edges. We
allow $G$ to have multiple edges but no loops. Let $\vec{a}=(a_1,a_2,\dots
,a_n)\in \ZZ^{n}$. An $m$-tuple $(b_{ij})_{(i,j)\in E}\in \mathbb{R}_{\geq 0}^m$
is called an \emph{$\vec{a}$-flow of $G$} if
\[\sum_{(i,j)\in E}b_{ij}(\vec{e}_i-\vec{e}_j)=\left(a_1,\dots,a_n, -\sum_{i=1}^n a_i\right),\]
where $\vec{e}_i$ is the standard basis vector in $\mathbb{R}^{n+1}$ with a one in the $i$th entry and zeroes elsewhere. The \emph{flow polytope $\FF_{G}(\vec{a})$ of $G$ with net flow $\vec{a}$} is defined as the set of all $\vec{a}$-flows of $G$.

In this paper we consider the following two graphs, see Figures~\ref{fig:PS} and \ref{fig:Car}:
\begin{itemize}
\item The \emph{Pitman-Stanley graph} $\PS_{n+1}$ is the graph with vertex set $[n+1]$ and edge set \[
\{(i,i+1): i=1,2,\dots, n\}\cup\{(i,n+1):i=1,2,\dots ,n-1\}.
\]
\item The \emph{Caracol graph} $\Car_{n+1}$ is the graph with vertex set $[n+1]$ and edge set
\[
\{(i,i+1):i=1,2,\dots,n\}\cup\{(1,i):i=3,4,\dots,n\}\cup\{(i,n+1):i=2,3,\dots,n-1\}.
\]
\end{itemize}

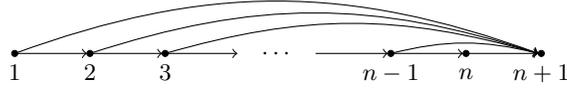
\begin{figure}
        \centering
        \begin{tikzpicture}
  \draw[->] (0,0) -- (0.96,0);
  \draw[->] (1,0) -- (1.96,0);
  \draw[->] (2,0) -- (2.96,0);
  \draw[->] (4,0) -- (4.96,0);
  \draw[->] (5,0) -- (5.96,0);
  \draw[->] (6,0) -- (6.96,0);
  \draw[->] (0,0) to [out=20,in=160] (6.96,0);
  \draw[->] (1,0) to [out=18,in=162] (6.96,0);
  \draw[->] (2,0) to [out=16,in=164] (6.96,0);
  \draw[->] (5,0) to [out=14,in=166] (6.96,0);
  \draw [fill] (0,0) circle [radius=0.04]; \node[below] at(0,0) {1};
  \draw [fill] (1,0) circle [radius=0.04]; \node[below] at(1,0) {2};
  \draw [fill] (2,0) circle [radius=0.04]; \node[below] at(2,0) {3};
  \draw [fill] (5,0) circle [radius=0.04]; \node[below] at(5,0) {$n-1$};
  \draw [fill] (6,0) circle [radius=0.04]; \node[below] at(6,-0.06) {$n$};
  \draw [fill] (7,0) circle [radius=0.04]; \node[below] at(7,0) {$n+1$};
  \node at (3.5,0) {$\cdots$};
\end{tikzpicture}
\caption{The Pitman-Stanley graph $\PS_{n+1}$.}
\label{fig:PS}
\end{figure}

\begin{figure}
        \centering
        \begin{tikzpicture}
  \draw[->] (0,0) -- (0.96,0);
  \draw[->] (1,0) -- (1.96,0);
  \draw[->] (2,0) -- (2.96,0);
  \draw[->] (4,0) -- (4.96,0);
  \draw[->] (5,0) -- (5.96,0);
  \draw[->] (6,0) -- (6.96,0);
  \draw[->] (1,0) to [out=18,in=162] (6.96,0);
  \draw[->] (2,0) to [out=16,in=164] (6.96,0);
  \draw[->] (5,0) to [out=14,in=166] (6.96,0);
  \draw[->] (0,0) to [out=-14,in=-166] (1.96,0);
  \draw[->] (0,0) to [out=-16,in=-164] (4.96,0);
  \draw[->] (0,0) to [out=-18,in=-162] (5.96,0);
  \draw [fill] (0,0) circle [radius=0.04]; \node[above] at(0,0) {1};
  \draw [fill] (1,0) circle [radius=0.04]; \node[above] at(1,0) {2};
  \draw [fill] (2,0) circle [radius=0.04]; \node[above] at(2,0) {3};
  \draw [fill] (5,0) circle [radius=0.04]; \node[below] at(5,0) {$n-1$};
  \draw [fill] (6,0) circle [radius=0.04]; \node[below] at(6,-0.06) {$n$};
  \draw [fill] (7,0) circle [radius=0.04]; \node[below] at(7,0) {$n+1$};
  \node at (3.5,0) {$\cdots$};
\end{tikzpicture}
        \caption{The caracol graph $\Car_{n+1}$.}
\label{fig:Car}
\end{figure}
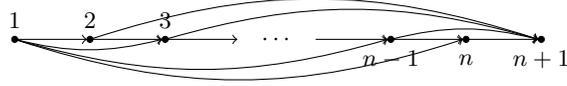

We note that the flow polytope $\FF_{\PS_{n+1}}(a_1,\dots,a_n)$ are affinely
equivalent to the polytope
\[
  \Pi_{n-1}(a_1,\dots,a_{n-1}):= \{(x_1,\dots,x_{n-1}): x_i\ge0, 
  x_1+\dots+x_{i}\le a_1+\dots+a_{i}, 1\le i\le n-1\},
\]
considered in \cite{PitmanStanley}. Pitman and Stanley \cite{PitmanStanley} found volume formulas for
certain polytopes, which can be restated as normalized volumns of 
 flow polytopes as follows:
\begin{align}
  \label{eq:PS11}
  \vol\FF_{\PS_{n+1}}(a,b^{n-2},d) &= a(a+(n-1)b)^{n-2},  \\
  \label{eq:PS12}
  \vol\FF_{\PS_{n+1}}(a,b^{n-3},c,d) &= a(a+(n-1)b)^{n-2} + (n-1)a(c-b)(a+(n-2)b)^{n-3}, \\
  \label{eq:PS13}
  \vol\FF_{\PS_{n+1}}(a,b^{n-m-2},c,0^{m-1},d) &= a\sum_{j=0}^m \binom nj (c-(m+1-j)b)^j(a+(n-1-j)b)^{n-j-2},
\end{align}
where $b^k$ means the sequence $b,b,\dots,b$ of $k$ $b$'s. We note that
$\vol\FF_{\PS_{n+1}}(a_1,\dots,a_{n})$ is independent of $a_n$.

In \cite{Benedetti2019}, Benedetti et al.~introduced combinatorial models called
gravity diagrams and unified diagrams to compute volumes of flow polytopes.
Using these models they showed
\begin{align}
  \label{eq:5}
\vol\FF_{\Car_{n+1}}(a^n) &=C_{n-2}a^n n^{n-2},\\
\vol\FF_{\Car_{n+1}}(a,b^{n-1})&=C_{n-2} a^{n-2}(a+(n-1)b)^{n-2},
\end{align}
where $C_k:=\frac{1}{k+1}\binom{2k}{k}$ is the $k$th Catalan number.

For a positive integer $k$ and a directed graph $G$ on $[n+1]$, let
$\GG(k)$ be the directed graph obtained from $G$ by adding a new
vertex $0$ and $k$ multiple edges $(0,i)$ for each $1\le i\le
n+1$. Then we define
\begin{equation}
  \label{eq:EG}
E_G(k) = \vol\mathcal{F}_{\GG(k)}(1,0^{n}).  
\end{equation}

In \cite{Benedetti2019}, Benedetti et al.~showed that $E_G(k)$ is a polynomial
function in $k$. Therefore we can consider the polynomial $E_G(x)$. They also
showed that these polynomials $E_G(x)$ have similar properties as Ehrhart
polynomials. For example, the leading coefficient of $E_G(x)$ is the normalized
volume of $\FF_G(1^n)$. For this reason, they called $E_G(x)$ an Ehrhart-like
polynomial. In the same paper they proved the following theorem.

\begin{theorem}\label{thm:PS}
We have
  \[
    E_{\PS_{n+1}}(k)=\frac{1}{kn-1}\binom{(k+1)n-2}{n}.  
  \]
\end{theorem}

Our main result is the following theorem, which was conjectured in \cite{Benedetti2019}.
\begin{theorem}\label{thm:main}
We have
\[E_{\Car_{n+1}}(k)=\frac{1}{kn+n-3}\binom{kn+2n-5}{n-1}\binom{n+k-3}{k-1}.\]  
\end{theorem}

In this paper we prove both Theorems~\ref{thm:PS} and \ref{thm:main}. 

The remainder of this paper is organized as follows.
 In Section~\ref{sec:const-term-ident} we use the Lidskii volume
formula to interpret $E_{\Car_{n+1}}(k)$ as a Kostant partition function, which is equal to the constant term of a
Laurent series.  In Section~\ref{sec:labeled-dyck-paths} we introduce labeled Dyck paths and show that the constant term
is equal to the number of certain labeled Dyck paths.  In Section~\ref{sec:cyclic-lemma} we enumerate these labeled Dyck
paths using a cyclic lemma.  In Section~\ref{sec:more-prop-label}
using our combinatorial models we  show the following volume formulas:
\begin{align}
\label{eq:PS3}
\vol\FF_{\PS_{n+1}}(a,b,c^{n-2}) & =(a+b-c)(a+b+(n-2)c)^{n-2}+(-b+c)(b+(n-2)c)^{n-2},\\
\label{eq:PS4}
    \vol\FF_{\PS_{n+1}}(a,b,c,d^{n-3})&=(a+b+c-2d)(a+b+c+(n-3)d)^{n-2} \\
\notag   & \qquad \qquad  -(b+c-2d)(b+c+(n-3)d)^{n-2}\\ 
\notag&\qquad\qquad -(n-1)a(c-d)(c+(n-3)d)^{n-3} ,\\
\label{eq:conj}
\vol\FF_{\Car_{n+1}}(a,b,c^{n-2}) &=C_{n-2}a^{n-1}(a+b(n-1))(a+b+c(n-2))^{n-3},
\end{align}
where \eqref{eq:conj} was conjectured by Benedetti et al. in \cite{Benedetti2019}.

\section{Constant term identities}
\label{sec:const-term-ident}

In this section we review the Lidskii volume formula and restate
Theorems~\ref{thm:PS} and \ref{thm:main} as constant term identities.

Let $G$ be a directed graph on $[n+1]$ and $\vec a\in\ZZ^n$. The \emph{Kostant partition function} $K_G(\vec{a})$ of $G$
at $\vec{a}$ is the number of integer points of $\FF_{G}(\vec{a})$, \emph{i.e.}, if $G$ has $m$ edges,
\[
  K_G(\vec{a})=|\FF_G(\vec{a}) \cap \ZZ^m|.
\]
We denote by $G|_{n}$ the restriction of $G$ to the vertices in $[n]$. Let $\outdeg(i)$ denote the out-degree of vertex
$i$ in $G$. The following formula, known as the Lidskii volume formula, allows us to express the (normalized) volume of
the flow polytope $\FF_G(\vec{a})$ in terms of Kostant partition functions, see \cite[Theorem~38]{Baldoni2008}.

\begin{theorem}[Lidskii volume formula] \label{thm:lidskii}
  Let $G$ be a connected directed graph on $[n+1]$ with $m$ directed edges, where every directed edge is of the form
  $(i, j)$ with $i<j$ and let $\vec{a}=(a_1,a_2,\dots ,a_n)\in \ZZ^{n}$. Denoting
  $\vec{t}=(t_1,\dots ,t_n):=(\outdeg(1)-1,\dots,\outdeg(n)-1)$, we have
\begin{equation*}
    \vol\FF_{G}(\vec a) = \sum_{\substack{|\vec s|=m-n\\ \vec{s}\ge \vec t}}\binom{m-n}{s_{1},s_{2},\dots ,s_{n}}a_{1}^{s_1}\dots a_{n}^{s_n}K_{G|_{n}}(\vec{s}-\vec{t}),
\end{equation*}
where the sum is over all sequences $\vec{s}=(s_1,\dots ,s_n)$ of nonnegative integers such that
$|\vec s|=s_1+\dots+s_n=m-n$ and $\vec{s}\geq \vec{t}$ in dominance order, \emph{i.e.}, $\sum_{i=1}^{k}s_{i} \geq \sum_{i=1}^{k}t_{i}$ for $k=1,2,\dots,n$.
\end{theorem}

Note that if $\vec a = (1,0^n)$ in Theorem~\ref{thm:lidskii} there is only one term in the sum giving the following
corollary, see \cite{Baldoni2008}, \cite{PitmanStanley}, or \cite[Corollary~1.4]{VolumesandEhrhart}.

\begin{corollary}
For a directed graph $G$ on $[n+1]$, we have
\begin{equation}
  \label{eq:outdeg}
\vol \FF_G(1,0^n) = K_G(p,1-\outdeg(2),1-\outdeg(3),\dots,1-\outdeg(n),0),  
\end{equation}
where $p=\outdeg(2)+\outdeg(3)+\cdots+\outdeg(n)-n+1$.
\end{corollary}

For a multivariate rational function $f(x_1,x_2,\dots,x_n)$ we denote by
$\CT_{x_i}f$ the constant term of the Laurant series expansion of $f$ with
respect to $x_i$ by considering other variables as constants. Since $\CT_{x_1}
f$ is a rational function in $x_2,\dots,x_n$, we can apply $\CT_{x_2}$ to it.
Repeating in this way the constant term $\CT_{x_{n}}\dots \CT_{x_{1}} f$ is
defined. We also define $[x_n^{a_n}\dots x_1^{a_1}]f$ to be the coefficient of the
monomial $x_n^{a_n}\dots x_1^{a_1}$ in the Laurent expansion of $f$ when
expanded in the variables $x_1, x_2,\dots,x_n$ in this order. Note that we have
\begin{equation}
  \label{eq:1}
[x_n^{a_n}\dots x_1^{a_1}]f = \CT_{x_{n}}\dots \CT_{x_{1}} \left(x_n^{-a_n}\dots x_1^{-a_1} f\right).  
\end{equation}

Let $G$ be a directed graph on $[n+1]$.
Then for  $\vec a=(a_1,\dots,a_{n}) \in \ZZ^{n}$
and $a_{n+1}=-(a_1+\cdots+a_n)$,  the Kostant partition function $K_G(\vec a)$ can be computed by
\begin{equation}
  \label{eq:KG}
K_G(\vec a) = [x_{n+1}^{a_{n+1}} \cdots x_1^{a_1}] \prod_{(i,j)\in E(G)}\left( 1-\frac{x_i}{x_j}\right)^{-1}.  
\end{equation}

Now we are ready to express $E_{\PS_{n+1}}(k)$ and $E_{\Car_{n+1}}(k)$ as
constant terms of Laurent series.
Throughout this paper the factor $(x_j-x_i)^{-1}$, where $i<j$, means the
Laurent expansion
\[
(x_j-x_i)^{-1} = \frac{1}{x_j} \left( 1-\frac{x_i}{x_j} \right)^{-1}
= \frac{1}{x_j}\sum_{l\ge0} \left(  \frac{x_i}{x_j}\right)^l.
\]

\begin{proposition}
\label{prop:car}
We have
\begin{align}
\label{eq:CT1}
E_{\PS_{n+1}}(k) &=\CT_{x_{n}}\dots \CT_{x_{1}}\prod_{i=1}^{n}(1-x_{i})^{-k}\prod _{i=1}^{n-1}(x_{i+1}-x_{i})^{-1},\\
\label{eq:CT2}
E_{\Car_{n+1}}(k) &=    \CT_{x_{n}}\dots \CT_{x_{1}}\frac{1}{x_{1}}\prod_{i=1}^{n}(1-x_{i})^{-k}\prod _{i=1}^{n-1}(x_{n}-x_{i})^{-1}\prod _{i=1}^{n-2}(x_{i+1}-x_{i})^{-1}.  
\end{align}

\end{proposition}
\begin{proof}
We will only prove \eqref{eq:CT2} since \eqref{eq:CT1} can be proved similarly.
Let $G=\Car_{n+1}$ and $H=\GG(k)$. Then $H$ is a graph with vertices 
$0,1,2,\dots,n+1$, and by \eqref{eq:EG} and \eqref{eq:outdeg},
\[
E_{\Car_{n+1}}(k)=K_H(p,1-\outdeg_H(1),1-\outdeg_H(2),\dots,1-\outdeg_H(n),0),
\]
where $p=\outdeg_H(1)+\outdeg_H(2)+\cdots+\outdeg_H(n)-n$.
Since $\outdeg_{H}(1)=n-1$, $\outdeg_{H}(n)=1$, and $\outdeg_{H}(i)=2$ for $2\le i\le n-1$, we can rewrite the above equation as
\[
E_{\Car_{n+1}}(k)=K_H(2n-4,2-n,(-1)^{n-2},0,0).
\]
Then, by \eqref{eq:KG}, we obtain
\begin{equation}
  \label{eq:3}
E_{\Car_{n+1}}(k)=  [x_{n-1}^{-1} \cdots x_2^{-1} x_1^{2-n} x_0^{2n-4}] \prod_{(i,j)\in E(H)}\left( 1-\frac{x_i}{x_j}\right)^{-1}.  
\end{equation}

Since every term in the expansion of 
\[
\prod_{(i,j)\in E(H)}\left( 1-\frac{x_i}{x_j}\right)^{-1}=
\prod_{i=1}^{n}\left( 1-\frac{x_0}{x_{i}} \right)^{-k}
\prod_{i=2}^n\left( 1-\frac{x_1}{x_{i}} \right)^{-1} \left( 1-\frac{x_i}{x_{n+1}} \right)^{-1}
\prod_{i=2}^{n-1}\left( 1-\frac{x_i}{x_{i+1}} \right)^{-1}
\]
is homogeneous of degree 0 in the variables $x_0,x_1,\dots,x_{n+1}$, we can set $x_{0}=1$ in \eqref{eq:3}.
Moreover, since every term in the expansion of $(1-x_i/x_{n+1})^{-1}$ has a negative power of $x_{n+1}$ except for the constant term $1$, we can omit 
the factors involving $x_{n+1}$ in \eqref{eq:3}. Then, by the same argument, we can also omit the factors involving $x_n$ in \eqref{eq:3} to obtain
\[
E_{\Car_{n+1}}(k) =[x_1^{2-n}x_2^{-1}\cdots x_{n-1}^{-1}] 
\prod_{i=1}^{n-1}\left( 1-\frac{1}{x_{i}} \right)^{-k}
\prod_{i=2}^{n-1}\left( 1-\frac{x_1}{x_{i}} \right)^{-1}
\prod_{i=2}^{n-2}\left( 1-\frac{x_i}{x_{i+1}} \right)^{-1}.
\]
By replacing $x_i$ by $x_{n-i}^{-1}$ for each $1\le i\le n-1$ we have
\[
E_{\Car_{n+1}}(k) = [x_{n-1}^{n-2} x_{n-2} \cdots x_1] 
\prod_{i=1}^{n-1}\left( 1-x_{i} \right)^{-k}
\prod_{i=1}^{n-2}\left( 1-\frac{x_i}{x_{n-1}} \right)^{-1}
\prod_{i=1}^{n-3}\left( 1-\frac{x_{i}}{x_{i+1}} \right)^{-1},
\]
which is equivalent to \eqref{eq:CT2} by \eqref{eq:1}.
\end{proof}

By Proposition \ref{prop:car}, we can restate Theorems~\ref{thm:PS} and \ref{thm:main} as follows. 

\begin{theorem}
\label{Thm:PS}
We have
\[
  \CT_{x_{n}}\dots \CT_{x_{1}}\prod_{i=1}^{n}(1-x_{i})^{-k}\prod _{i=1}^{n-1}(x_{i+1}-x_{i})^{-1} = \frac{1}{kn-1}\binom{(k+1)n-2}{n}.
\]
\end{theorem}

\begin{theorem}
\label{Thm:Car}
We have
\begin{multline*}
    \CT_{x_{n}}\dots \CT_{x_{1}}\frac{1}{x_{1}}\prod_{i=1}^{n}(1-x_{i})^{-k}\prod _{i=1}^{n-1}(x_{n}-x_{i})^{-1}\prod _{i=1}^{n-2}(x_{i+1}-x_{i})^{-1}\\
=\frac{1}{k(n+1)+n-2}\binom{kn+k+2n-3}{n}\binom{n+k-2}{k-1}.
\end{multline*}
\end{theorem}

\section{Labeled Dyck Paths}
\label{sec:labeled-dyck-paths}

In this section we give combinatorial meanings to the constant terms in Theorems~\ref{Thm:PS} and
\ref{Thm:Car} using labeled Dyck paths.

 A \emph{Dyck path of length $2n$} is a lattice path from $(0,0)$ to $(2n,0)$ consisting of
 \emph{up-steps} $(1,1)$ and \emph{down-steps} $(1,-1)$ lying on or above the line $y=0$.  The set
 of Dyck paths of length $2n$ is denoted by $\Dyck_n$.

 Let $k$ be a positive integer. A \emph{$k$-labeled Dyck path} is a Dyck path with a labeling on the
 down-steps such that the label of each down-step is an integer $0\le i\le k$ and the labels of any
 consecutive down-steps are in weakly decreasing order, see Figure~\ref{fig:LD}.  A \emph{doubly
   $k$-labeled Dyck path} is a $k$-labeled Dyck path together with an additional labeling on the
 down-steps labeled $0$ and the up-steps with integers from $\{1,2,\dots,k\}$ such that the additional labels on these
 steps are weakly increasing, see Figure~\ref{fig:DLD}.  

 We denote by $\LD_n(k)$ (resp.~$\DLD_n(k)$) the set of $k$-labeled Dyck paths (resp.~doubly
 $k$-labeled Dyck paths) of length $2n$. We also denote by $\LD_n(k,d)$ the set of $k$-labeled Dyck
 paths of length $2n$ with exactly $d$ down-steps labeled $0$.

 \begin{figure}
   \centering
\begin{tikzpicture}[scale=0.6]
  \draw[black!20] (0,0) grid (20,6);
  \draw[line width=1.5pt] (0,0)--
  ++(1,1)--++(1,-1)--++(1,1)--++(1,1)--  ++(1,1)--  ++(1,1)--  ++(1,-1)-- ++(1,1)--  ++(1,-1)--  ++(1,-1)--  ++(1,1)--  ++(1,1)--  ++(1,1)-- ++(1,-1)--  ++(1,-1)--  ++(1,-1)--  ++(1,-1)--  ++(1,1)--  ++(1,-1)--  ++(1,-1);
  \draw (0,0)--(20,0);
  \node[above right] at (1,0) {0};
  \node[above right] at (6,3) {5};
  \node[above right] at (8,3) {3};
  \node[above right] at (9,2) {0};
  \node[above right] at (13,4) {4};
  \node[above right] at (14,3) {2};
  \node[above right] at (15,2) {0};
  \node[above right] at (16,1) {0};
  \node[above right] at (18,1) {1};
  \node[above right] at (19,0) {1};
  \node[below] at (0,0) {$(0,0)$};
  \node[below] at (20,0) {$(20,0)$};
\end{tikzpicture}
\caption{A $5$-labeled Dyck path of length $20$. The labels of every consecutive down-steps must be in weakly
decreasing order.}
\label{fig:LD}
 \end{figure}
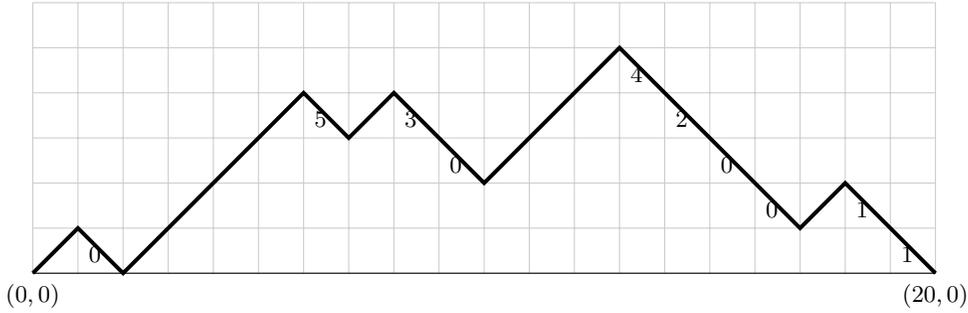

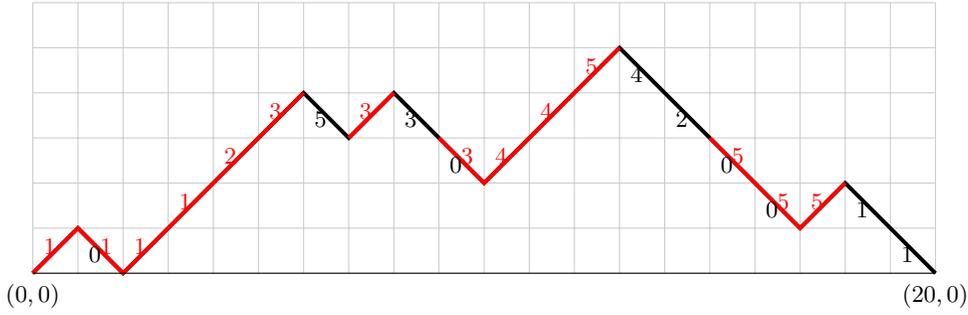
\begin{figure}
\begin{tikzpicture}[scale=0.6]
  \draw[black!20] (0,0) grid (20,6);
  \draw[line width=1.5pt] (0,0)--
  ++(1,1)--++(1,-1)--++(1,1)--++(1,1)--  ++(1,1)--  ++(1,1)--  ++(1,-1)-- ++(1,1)--  ++(1,-1)--  ++(1,-1)--  ++(1,1)--  ++(1,1)--  ++(1,1)-- ++(1,-1)--  ++(1,-1)--  ++(1,-1)--  ++(1,-1)--  ++(1,1)--  ++(1,-1)--  ++(1,-1);
  \draw[line width=1.5pt, red] (0,0)--(1,1)--(2,0)--(6,4) (7,3)--(8,4) (9,3)--(10,2)--(13,5) (15,3)--(17,1)--(18,2);
  \draw (0,0)--(20,0);
  \node[above right] at (1,0) {0};
  \node[above right] at (6,3) {5};
  \node[above right] at (8,3) {3};
  \node[above right] at (9,2) {0};
  \node[above right] at (13,4) {4};
  \node[above right] at (14,3) {2};
  \node[above right] at (15,2) {0};
  \node[above right] at (16,1) {0};
  \node[above right] at (18,1) {1};
  \node[above right] at (19,0) {1};
  \node[below right,red] at (0,1) {1};
  \node[below left,red] at (2,1) {1};
  \node[below right,red] at (2,1) {1};
  \node[below right,red] at (3,2) {1};
  \node[below right,red] at (4,3) {2};
  \node[below right,red] at (5,4) {3};
  \node[below right,red] at (7,4) {3};
  \node[below left,red] at (10,3) {3};
  \node[below right,red] at (10,3) {4};
  \node[below right,red] at (11,4) {4};
  \node[below right,red] at (12,5) {5};
  \node[below left,red] at (16,3) {5};
  \node[below left,red] at (17,2) {5};
  \node[below right,red] at (17,2) {5};
  \node[below] at (0,0) {$(0,0)$};
  \node[below] at (20,0) {$(20,0)$};
\end{tikzpicture}
\caption{A doubly $5$-labeled Dyck path of length $20$. 
The down-steps labeled $0$ and the up-steps are the red steps and their additional labels are
written in red. The red labels must be in weakly increasing order.}
\label{fig:DLD}
\end{figure}

A \emph{multiset} is a set with repetitions allowed.  Let $\multiset{n}{m}:=\binom{n+m-1}{m}$. Then
$\multiset{n}{m}$ is the number of multisets with $m$ elements taken from $[n]$. Equivalently,
$\multiset{n}{m}$ is the number of nonnegative integer solutions $(a_{1},a_{2},\dots ,a_{n})$ to
$a_{1}+a_{2}+\dots +a_{n}=m$ and also the number of $m$-tuples $(i_{1},i_{2},\dots ,i_{m})$ of nonnegative integers satisfying $1\leq i_{1} \leq i_{2} \leq \dots \leq i_{m} \leq n$.

The following proposition is immediate from the definitions of $\DLD_n(k)$ and $\LD_n(k,d)$.

\begin{proposition}
We have
\label{prop:DLD&LD}
\begin{align*}
    |\DLD_n(k)|=\sum_{d=0}^{n}|\LD_n(k,d)| \Multiset{k}{n+d}.
\end{align*}
\end{proposition}

We now show that the constant terms in Theorems~\ref{Thm:PS} and \ref{Thm:Car} have the following combinatorial
interpretations.

\begin{theorem}
We have
\label{Thm:PS=LD}
\begin{align*}
    \CT_{x_{n}}\dots \CT_{x_{1}}\prod_{i=1}^{n}(1-x_{i})^{-k}\prod _{i=1}^{n-1}(x_{i+1}-x_{i})^{-1} = |\LD_{n-1}(k,0)|.
\end{align*}
\end{theorem}
\begin{proof}
Consider that we choose $x_i^{a_{i1}}x_i^{a_{i2}}\cdots x_i^{a_{ik}}$ in $(1-x_{i})^{-k}=(1+x_{i}+x_{i}^{2}+\cdots)\cdots (1+x_{i}+x_{i}^{2}+\cdots)$ for $i=1,2,\dots ,n$ and we choose $x_{i}^{b_{i}}/x_{i+1}^{b_{i}+1}$ in $(x_{i+1}-x_{i})^{-1}=1/x_{i+1}+x_{i}/x_{i+1}^{2}+x_{i}^{2}/x_{i+1}^{3}+\cdots$ for $i=1,2,\dots,n-1$. Then $\prod_{i=1}^{n}(1-x_{i})^{-k}\prod _{i=1}^{n-1}(x_{i+1}-x_{i})^{-1}=\sum\prod_{i=1}^{n}(x_{i}^{a_{i1}+\dots+ a_{ik}+b_{i}-b_{i-1}-1})$, where the sum is over all nonnegative integers $a_{ij},b_{i}$ for $1\leq i\leq n$ and $1\leq j \leq k$ with $b_{0}=-1$ and $b_{n}=0$. Hence the left-hand side is the number of the nonnegative integer solutions to the equations $a_{i1}+\dots +a_{ik}=b_{i-1}-b_{i}+1$ for $i=1,2,\dots,n$ with $b_{0}=-1,b_{n}=0$. If we set $r_{i}=b_{i-1}-b_{i}+1$, so that $r_1+r_2+\dots +r_j=j+b_0-b_j\leq j-1$, then the number of solutions is $\sum\prod_{i=1}^{n}\multiset{k}{r_{i}}$, where the sum is over all nonnegative integers $r_{1},\dots,r_{n}$ with $r_{1}+\dots +r_{j}\leq j-1$ for $j=1,2,\dots, n-1$ and $r_{1}+\dots +r_{n}=n-1$. For such an $n$-tuple $(r_1,\dots,r_n)$, let $D$ be the Dyck path of length $2(n-1)$ such that the number of consecutive down-steps after the $i$th up-step is $r_{i+1}$ for $i=1,\dots,n-1$. The map $(r_1,\dots,r_n) \mapsto D$ is a bijection from the set of $n$-tuples satisfying the above conditions to $\Dyck_{n-1}$. Under this correspondence, $\prod_{i=1}^{n}\multiset{k}{r_i}$ is the number of $k$-labeled Dyck paths in $\LD_{n-1}(k,0)$ whose underlying Dyck path is $D$. Therefore we obtain the result.
\end{proof}

\begin{theorem}
We have
\label{Thm:Car=DLD}
\begin{align*}
     \CT_{x_{n}}\dots \CT_{x_{1}}\frac{1}{x_{1}}\prod_{i=1}^{n}(1-x_{i})^{-k}\prod _{i=1}^{n-1}(x_{n}-x_{i})^{-1}\prod _{i=1}^{n-2}(x_{i+1}-x_{i})^{-1} = |\DLD_{n-1}(k)|.
\end{align*}
\end{theorem}
\begin{proof}
Similarly to the previous theorem, considering $x_i^{a_{i1}}x_i^{a_{i2}}\cdots x_i^{a_{ik}}$ in
$(1-x_{i})^{-k}$ and $x_{i}^{b_{i}}/x_{i+1}^{b_{i}+1}$ in $(x_{i+1}-x_i)^{-1}$ and
$x_{i}^{c_{i}}/x_{n}^{c_{i}+1}$ in $(x_n-x_i)^{-1}$, we get that the left-hand side is the number of
the nonnegative integer solutions to the equations $a_{i1}+\dots +a_{ik}+b_{i}+c_{i}=1+b_{i-1}$ for
$i=1,2,\dots,n-1$ and $a_{n1}+\dots+ a_{nk}+b_{n}+c_{n}=n-1+c_{1}+\dots +c_{n-1}$ with
$b_{0}=b_{n-1}=b_{n}=c_{n}=0$. If we set $r_{i}=b_{i-1}-b_{i}+1$ for $i=1,2,\dots,n-1$, then the
number of solutions is
$\sum_{r_i,c_i}\prod_{i=1}^{n-1}\multiset{k}{r_{i}-c_i}\multiset{k}{n-1+c_{1}+\dots +c_{n-1}}$
 where the sum is over all nonnegative integers $r_i,c_i$ for $i=1,2,\dots, n-1$ with $r_{1}+\dots +r_{j}\leq j$ for $j=1,2,\dots, n-2$ and $r_{1}+\dots +r_{n-1}=n-1$. For such an $n$-tuple $(r_1,\dots,r_n)$, let $D$ be the Dyck path of length $2(n-1)$ such that the number of consecutive down-steps after the $i$th up-step is $r_{i}$ for $i=1,\dots,n-1$. The map $(r_1,\dots,r_n) \mapsto D$ is a bijection from the set of $n$-tuples satisfying the above conditions to $\Dyck_{n-1}$. Regard $\multiset{k}{r-c}$ as the number of $r$-tuples $(i_1,\dots,i_r)$ of integers with $k\geq i_1 \geq \dots \geq i_{r-c} \geq 1$ and $ i_{r-c+1} =\dots=i_r= 0$. Then $\sum_{c_i}\prod_{i=1}^{n-1}\multiset{k}{r_{i}-c_i}\multiset{k}{n-1+c_{1}+\dots +c_{n-1}}$, where the sum is over all nonnegative integers $c_i$ for $i=1,2,\dots,n-1$ with $c_i \leq r_i$, is the number of doubly $k$-labeled Dyck paths whose underlying Dyck path is $D$. Therefore we obtain the result.
\end{proof}

Note that by Proposition~\ref{prop:DLD&LD}, we can compute the constant terms in
Theorems~\ref{Thm:PS=LD} and \ref{Thm:Car=DLD} if we have a formula for the cardinality $|\LD_n(k,d)|$.
Therefore our next step is to find this number.

\section{A cyclic lemma}
\label{sec:cyclic-lemma}

Let $\LD_{n}(k;a_0,a_1,\dots,a_k)$ denote the set of $k$-labeled Dyck paths of length $2n$ 
such that the number of down-steps with label $i$ is $a_i$ for $0\le i\le k$. 
In this section we prove the following theorem using a cyclic lemma. 

\begin{theorem}
\label{thm:dyck}
We have
\[
|\LD_n(k;a_0,a_1,\dots,a_k)|=\frac{1}{n+1}\prod_{i=0}^{k}\Multiset{n+1}{a_i}.
\]
\end{theorem}

\begin{remark}
A \emph{parking function} of length $n$ is a tuple $(p_1,p_2,\dots,p_n)\in \ZZ^n_{>0}$ with a condition that $q_i\leq i$ for $i=1,2,\dots,n$ where $(q_1,q_2,\dots,q_n)$ is the rearrangement of $(p_1,p_2,\dots,p_n)$ in weakly increasing order. Let $PF_n$ be the set of parking function of length $n$.
There is a well-known bijection between $PF_n$ and $n$-labeled Dyck paths of length $2n$ which the number of each label from $1$ to $n$ equals $1$. Thus, using Theorem \ref{thm:dyck}, we have
\[|PF_n|=|\LD_n(n;0,1,1,\dots,1)|=(n+1)^{n-1}.\]
\end{remark}

\begin{remark}
  Recently, Yip \cite[Theorem 3.18]{Yip2019} considered a set
  $\mathcal{T}_k(n,i)$ of certain labeled Dyck paths and found a simple
  formula for its cardinality using a cyclic lemma. Using our notation, this set can be written
  \[
\mathcal{T}_k(n,i) = \bigcup_{a_0+\dots+a_{k-1}=n-i}\Dyck_n(k+i-1;a_0,\dots,a_{k-1},1^i).
  \]
The proof of Theorem~\ref{thm:dyck} in this section is essentially the same as
that in \cite[Theorem 3.18]{Yip2019}.
\end{remark}
 
A \emph{$k$-labeled Dyck word} of length $2n$ is a sequence $w = w_1\dots w_{2n}$ of letters in $\{U,D_0,D_1,\dots,D_k\}$ 
satisfying the following conditions:
\begin{itemize}
\item The number of $U$'s is equal to $n$. 
\item For any prefix $w_1\dots w_j$, the number of $U$'s is greater than or equal to the total number of $D_i$'s for
  $0\le i\le k$.
\item The labels of any consecutive $D_i$'s are in weakly decreasing order, i.e.,
if $w_i=D_a$ and $w_{i+1}=D_b$, then $a\ge b$.
\end{itemize}

Replacing each up step by $U$ and each down step labeled $i$ by $D_i$ is an obvious bijection from $k$-labeled Dyck
paths to $k$-labeled Dyck words. For example, the $k$-labeled Dyck word corresponding to
the $k$-labeled Dyck path in Figure~\ref{fig:LD} is \begin{equation}
  \label{eq:w}
UD_0UUUUD_5UD_3D_0UUUD_4D_2D_0D_0UD_1D_1.  
\end{equation}

From now on, we will identify $k$-labeled Dyck paths with $k$-labeled Dyck words using this bijection. Note that
$\LD_n(k;a_0,a_1,\dots,a_k)$ is then the set of $k$-labeled Dyck words of length $2n$ in which the number of $D_i$'s is
equal to $a_i$ for $0\le i\le k$. We can count such words by using a well-known cyclic argument. We first need another
definition.

An \emph{extended $k$-labeled word} of length $2n+1$ is a sequence $w = w_1\dots w_{2n+1}$ of letters in
$\{U,D_0,D_1,\dots,D_k\}$ with $w_1=U$ and  exactly $n+1$ $U$'s that satisfies the third condition of a $k$-labeled Dyck
word: the labels of any consecutive $D_i$'s are in weakly decreasing order, i.e.,
if $w_i=D_a$ and $w_{i+1}=D_b$, then $a\ge b$.  The set of extended
$k$-labeled words of length $2n+1$ is denoted by $\EW_n(k)$.

Let $w = w_1\dots w_{2n+1}\in \EW_n(k)$. We define the integer $\ind(w)$ using the
following algorithm. Here, $w = w_1\dots w_{2n+1}$ is cyclically ordered, which means that $w_1$ is followed by $w_2$,
$w_2$ is followed by $w_3$, and so on, and $w_{2n+1}$ is followed by $w_1$.
\begin{itemize}
\item Find a letter $U$ followed by a $D_i$ for some $0\le i\le k$ in cyclic order and delete this pair $U$ and $D_i$
  from $w$. Repeat this until there is only one letter left, which must be $U$.
\item If the remaining $U$ is the $j$th $U$ in the original word $w$ then define $\ind(w)=j$.
\end{itemize}
We also define the \emph{shifting operator} $s:\EW_n(k)\to \EW_n(k)$ by
\[
  s(w): = w_{i}w_{i+1}\cdots w_{2n+1} w_1 \dots w_{i-1},
\]
where $i$ is the largest integer with $w_i=U$.

\begin{example}
Let $w=UD_1D_0UUD_1U \in \EW_3(1)$. Then by the algorithm,
\[
\underline{UD_1}D_0UUD_1U \rightarrow D_0U\underline{UD_1}U \rightarrow \underline{D_0}U\underline{U} \rightarrow U,
\]
we get $\ind(w)=2$ since the remaining $U$ is the second $U$ in $w$. 
\end{example}

Since the above algorithm treats the word $w$ cyclically one can easily see that the following lemma holds. 

\begin{lemma}\label{lem:cyclic}
For any element $w\in \EW_n(k)$, we have
  \[
    \ind(s(w)) \equiv \ind(w) + 1 \mod n+1.
  \]
\end{lemma}

Observe that for $w=w_1w_2\dots w_{2n+1}\in \EW_n(k)$ we have $w_1\dots w_{2n}\in\Dyck_n(k)$ if and only if
$\ind(w)=n+1$. Therefore, by Lemma~\ref{lem:cyclic}, for each $w\in \EW_n(k)$ there is a unique integer $0\le j\le n$
such that $s^j(w) = w'U$ for some $k$-labeled Dyck word $w'$ of length $2n$. This defines a map
$p: \EW_n(k)\to\Dyck_n(k)$ sending $w$ to $p(w)=w'$. Again, by Lemma~\ref{lem:cyclic}, this is a $(n+1)$-to-$1$ map.
Note that $w$ and $w'$ have the same the number of steps $D_i$ for each $0\le i\le k$. We have proved the following
proposition. 

\begin{proposition}\label{prop:n-to-1}
There is an $(n+1)$-to-$1$ map $p: \EW_n(k)\to\Dyck_n(k)$ preserving the number of $D_i$'s for $0\le i\le k$.
\end{proposition}

We now can prove Theorem~\ref{thm:dyck} easily.

\begin{proof}[Proof of Theorem~\ref{thm:dyck}]
By Proposition~\ref{prop:n-to-1}, $(n+1)|\LD_n(k;a_0,a_1,\dots,a_k)|$ is the number of
  elements $w\in \EW_n(k)$ in which $D_i$ appears $a_i$ times for $0\le i\le k$. Since consecutive $D_i$'s are always
  ordered according to their subscripts, such elements $w$ are obtained from the sequence $U\dots U$ of $n+1$ $U$'s by
  inserting $a_i$ $D_i$'s after $U$'s in $\multiset{n+1}{a_i}$ ways for $0\le i\le k$ independently. Thus we have
  \[
    (n+1)|\LD_n(k;a_0,a_1,\dots,a_k)|=\prod_{i=0}^{k}\Multiset{n+1}{a_i},
\]
which completes the proof.
\end{proof}

As corollaries we obtain formulas for $|\LD_{n}(k,d)|$ and $|\DLD_{n-1}(k)|$.

\begin{corollary}
\label{cor:LD}
We have
\[
   |\LD_{n}(k,d)|=\frac{1}{n+1}\Multiset{n+1}{d} \Multiset{k(n+1)}{n-d}.
\]
\end{corollary}
\begin{proof}
  By Theorem~\ref{thm:dyck},
  \[
    |\LD_{n}(k,d)|=\frac{1}{n+1} \Multiset{n+1}{d} \sum_{a_1+\dots+a_k=n-d}\prod_{i=1}^{k}\Multiset{n+1}{a_i}.
  \]
  The above sum is equal to the number of $k$-tuples $(A_1,\dots,A_k)$ of multisets such that
  each element $x\in A_i$ satisfies $(n+1)(i-1)+1\le x\le  (n+1)i$ and $\sum_{i=1}^k |A_i| = n-d$. Since such a
  $k$-tuple is completely determined by $A:=A_1\cup \cdots\cup A_k$, the sum is
  equal to $\multiset{k(n+1)}{n-d}$, the number of multisets of size $n-d$ whose elements are in $[k(n+1)]$. Thus we
  obtain the formula.
\end{proof}

\begin{corollary}
\label{cor:DLD}
We have
\begin{align*}
   |\DLD_{n-1}(k)|=\frac{1}{k(n+1)+n-2}\binom{kn+k+2n-3}{n}\binom{n+k-2}{k-1}.
\end{align*}
\end{corollary}
\begin{proof}
We will use the standard notation in hypergeometric series, see for example \cite[Chapter 2]{AAR_SP}. By Proposition \ref{prop:DLD&LD} and Corollary~\ref{cor:LD},
\begin{align}
\notag    |\DLD_{n-1}(k)|&=\sum_{d=0}^{n-1}|\LD_{n-1}(k,d)|\Multiset{k}{d+n-1}
    \\ \notag &=\sum_{d=0}^{n-1}\frac{1}{n}\Multiset{kn}{n-d-1}\Multiset{n}{d}\Multiset{k}{d+n-1}
    \\ \label{eq:2} &=\frac{(kn+n-2)!(k+n-2)!}{n!(kn-1)!(n-1)!(k-1)!}\Hyper21{-n+1, k+n-1}{-kn-n+2}{1} .
\end{align}
By the Vandermonde summation formula \cite[Corollary~2.2.3]{AAR_SP}
\[
  \Hyper21{-n,b}{c}{1} = \frac{(c-b)_n}{(c)_n},
\]
we have
\begin{equation}
  \label{eq:4}
\Hyper21{-n+1, k+n-1}{-kn-n+2}{1}  = \frac{(-kn-2n-k+3)_{n-1}}{(-kn-n+2)_{n-1}}
=\frac{(kn+2n+k-3)!}{(kn+n+k-2)!}\frac{(kn-1)!}{(kn+n-2)!}. 
\end{equation}
By \eqref{eq:2} and \eqref{eq:4}  we obtain the result.
\end{proof}

The constant term identities in Theorems~\ref{Thm:PS} and \ref{Thm:Car} follow
immediately from Theorems~\ref{Thm:PS=LD}, \ref{Thm:Car=DLD} and Corollary~\ref{cor:DLD}.
This completes the proof of Theorems~\ref{thm:PS} and \ref{thm:main} in the introduction.

\section{More Properties of Labeled Dyck Paths}
\label{sec:more-prop-label}

 In this section we find volumes of flow polytopes of Pitman--Stanley graph $\PS_{n+1}$ and Caracol
graph $\Car_{n+1}$ for certain flow vectors using Lidskii's formula and $k$-labeled Dyck prefixes.

A \emph{$k$-labeled Dyck prefix} is the part of a $k$-labeled Dyck path from $(0,0)$ to $(a,b)$ for some point $(a,b)$
in the path.  The set of $k$-labeled Dyck prefixes from $(0,0)$ to $(2n-i,i)$ is denoted by $\Dyck_{n,i}$.  We also
denote by $\Dyck_{n,i}(k;a_0,a_1,\dots,a_k)$ the set of $k$-labeled Dyck prefixes in $\Dyck_{n,i}$ such that the number
of down-steps labeled $j$ is $a_j$ for $0\le j\le k$.

Recall that $\Dyck_n(k)$ is in bijection with the set of $k$-Dyck words of length $2n$.  Therefore one can consider an element in $\Dyck_{n,i}$ as a $k$-Dyck word of length $2n$ whose last $i$ letters are
$D_0$'s. See Figure~\ref{fig:LD2}.

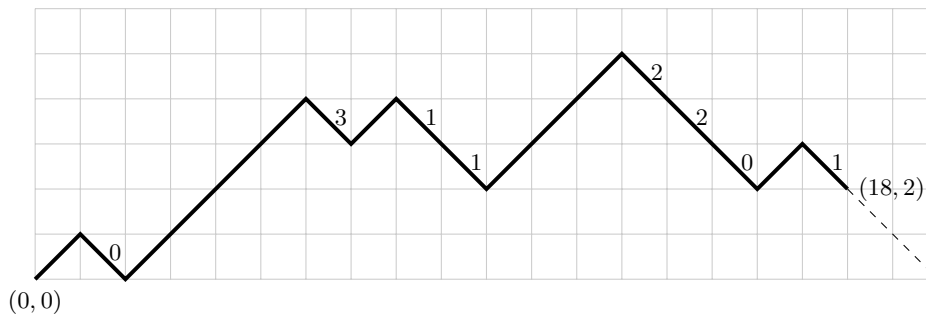
\begin{figure}
\begin{tikzpicture}[scale=0.6]
\def\blnum(#1,#2)[#3]{\node[below right] at (#1+0.4,#2) {#3};}
  \draw[black!20] (0,0) grid (20,6);
  \draw[line width=1.5pt] (0,0)--
  ++(1,1)--++(1,-1)--++(1,1)--++(1,1)--  ++(1,1)--  ++(1,1)--  ++(1,-1)-- ++(1,1)--  ++(1,-1)--  ++(1,-1)--  ++(1,1)--  ++(1,1)--  ++(1,1)-- ++(1,-1)--  ++(1,-1)--  ++(1,-1)--  ++(1,1)--  ++(1,-1);
\draw[dashed] (18,2)--  ++(1,-1)--  ++(1,-1);
\blnum(1,1)[0];
\blnum(6,4)[3];
\blnum(8,4)[1];
\blnum(9,3)[1];
\blnum(13,5)[2];
\blnum(14,4)[2];
\blnum(15,3)[0];
\blnum(17,3)[1];
\node[below] at (0,0) {$(0,0)$};
\node[right] at (18,2) {$(18,2)$};
\end{tikzpicture}
\caption{An element of $\Dyck_{20,2}(3;2,3,2,1)$ whose steps are drawn in solid segments.  By appending it with the two
  dashed down-steps, this element can be considered as an element in $\Dyck_{20}(3)$. This path can be expressed as
  $UD_0UUUUD_3UD_1D_1UUUD_2D_2D_0UD_1D_0D_0$, where the last two $D_0$ steps correspond to the dashed down-steps.}
\label{fig:LD2}
\end{figure}

Now we find the cardinality of $\Dyck_{n,i}(k;a_0,a_1,\dots,a_k)$.

\begin{lemma}
\label{lem:dyck_n,i}
We have
\[
|\Dyck_{n,i}(k;a_0,a_1,\dots,a_k)|=\dfrac{i+1}{n+1}\prod_{j=0}^{k}\Multiset{n+1}{a_j}.
\]
\end{lemma}
\begin{proof}
Let $\EW_{n,i}(k)$ be the set of words $w=w_1\dots w_{2n-i+1}$ of letters in
$\{U,D_0,D_1,\dots,D_k\}$ with $w_1=U$ and  exactly $n+1$ $U$'s that satisfies the third condition of a $k$-labeled Dyck
word: the labels of any consecutive $D_i$'s are in weakly decreasing order, i.e.,
if $w_i=D_a$ and $w_{i+1}=D_b$, then $a\ge b$.
For $w \in \EW_{n,i}(k)$, an \emph{index candidate} of $w$ is an integer $j$
satisfying the following condition:
\begin{itemize}
\item Find a letter $U$ followed by a $D_i$ for some $0\le i\le k$ in cyclic order and delete this pair $U$ and $D_i$
  from $w$. Repeat this until there are $i+1$ letters left, which must be all $U$'s. Then the $j$th $U$ in the original word $w$ is one of the
  remaining $U$'s.
\end{itemize}
Note that there are  $i+1$ index candidates for any $w\in \EW_{n,i}(k)$. 

Let $\EW_{n,i}'(k)$ be the set of words obtained from a word $w\in\EW_{n,i}(k)$ by adding $i$ $D_0$'s to the left of the
$j$th $U$ in $w$ for an index candidate $j$ of $w$.  Note that $\EW_{n,i}'(k)$ is a subset of $\EW_{n}(k)$ which is defined
in Section \ref{sec:cyclic-lemma}. Thus every $w'\in \EW_{n,i}'(k)$ has length $2n+1$ and the unique index $\ind(w')$
exists. Then by the map $p$ defined in Proposition \ref{prop:n-to-1}, there is an $(n+1)$-to-$1$ map from
$\EW_{n,i}'(k)$ to $\Dyck_{n,i}(k)$. Since there are $(i+1)$ ways to choose an index candidate for $w\in \EW_{n,i}(k)$,
we have
\[
(i+1)\prod_{j=0}^k\Multiset{n+1}{a_j}=(n+1)|\Dyck_{n,i}(k;a_0,a_1,\dots,a_k)|,
\]
which completes the proof.
\end{proof}

\subsection{Volumes of flow polytopes for the Pitman--Stanley graph.}
\label{subsec:PS}
Recall that $\vol\FF_{\PS_{n+1}}(a^n)$ and $\vol\FF_{\PS_{n+1}}(a,b^{n-1})$ were
computed in \cite{Benedetti2019} and \cite{PitmanStanley}.  In this subsection
using Lemma~\ref{lem:dyck_n,i} we compute
$\vol\FF_{\PS_{n+1}}(a_1,\dots ,a_k,b^{n-k})$ for $0\le k\le 3$.  For
simplicity, we will consider $\PS_{n+2}$ instead of $\PS_{n+1}$.

Note that $\PS_{n+2}$ has $2n+1$ edges and $\vec{t}:=(\outdeg(1)-1,\dots,\outdeg(n+1)-1)=(1,1,\dots,1,0)$.  Since
$\PS_{n+2}|_{n+1}$ is the path graph on $[n+1]$ with edges $(i, i+1)$ for $1\le i\le n$, one can easily see that
$K_{\PS_{n+2}|_{n+1}}(\vec{s}-\vec{t})=1$ for any sequence $\vec{s}\ge \vec t$.  Moreover, if
$\vec s=(s_1,\dots,s_{n+1})\ge \vec t$, then $s_{n+1}=0$.  Thus Lidskii's formula (Theorem~\ref{thm:lidskii}) implies
\[
    \vol\FF_{\PS_{n+2}}(a_1,\dots ,a_{n+1})=\sum_{\substack{s_1+\dots+s_{n}=n\\ (s_1,\dots,s_{n})\ge
        (1^{n})}}\binom{n}{s_{1},s_{2},\dots ,s_{n}}a_{1}^{s_1}\dots a_{n}^{s_{n}} .
\]
Thus, we have
\begin{align}
\notag & \vol\FF_{\PS_{n+2}}(a_1,\dots ,a_k,b^{n-k+1}) \\
\notag &=\sum_{\substack{s_1+\dots+s_{n}=n\\ (s_1,\dots,s_{n})\ge (1^{n})}}
\binom{n}{s_{1},s_{2},\dots ,s_{n}}a_{1}^{s_1}\dots a_{k}^{s_{k}} b^{s_{k+1}+\dots+s_{n}}\\
\notag   &=\sum_{m=0}^{n} \sum_{\substack{s_1+\dots+s_{k}=m\\ (s_1,\dots,s_{k})\ge (1^k)}}
\sum_{\substack{s_{k+1}+\dots+s_{n}=n-m\\ (m,s_{k+1},\dots,s_{n})\ge (k,1^{n-k})}}  
\binom{n}{s_{1},s_{2},\dots ,s_{n}}a_{1}^{s_1}\dots a_{k}^{s_{k}} b^{n-m}\\
\label{eq:PSA}  &=\sum_{m=0}^{n} \binom nm b^{n-m} A_{k,m}(a_1,\dots,a_k)B_{n,k,m},
\end{align}
where
\begin{align*}
A_{k,m}(a_1,\dots,a_k) &= \sum_{\substack{s_1+\dots+s_{k}=m\\ (s_1,\dots,s_{k})\ge (1^k)}}
\binom{m}{s_{1},s_{2},\dots ,s_{k}}a_{1}^{s_1}\dots a_{k}^{s_{k}}, \\
B_{n,k,m}&= \sum_{\substack{s_{k+1}+\dots+s_{n}=n-m\\ (m,s_{k+1},\dots,s_{n})\ge (k,1^{n-k})}}  
\binom{n-m}{s_{k+1},\dots ,s_{n}}.
\end{align*}
The following lemma shows that $B_{n,k,m}$ has a simple formula.

\begin{lemma}
We have
\label{lem:B}
  \[
    B_{n,k,m} = (m-k+1)(n-k+1)^{n-m-1}.
  \]
\end{lemma}
\begin{proof}
  For a sequence $(s_{k+1},\dots,s_n)$ of nonnegative integers, we have
  $s_{k+1}+\dots+s_n=n-m$ and $(m,s_{k+1},\dots,s_n)\ge (k,1^{n-k})$ if and only if
  $UD^{s_n}UD^{s_{n-1}}\dots UD^{s_{k+1}}U^kD^m$ is a Dyck path from $(0,0)$ to $(2n,0)$, or equivalently,
$UD^{s_n}UD^{s_{n-1}}\dots UD^{s_{k+1}}$ is a Dyck prefix from $(0,0)$ to $(2n-m-k,m-k)$. Moreover, if such a sequence
  $(s_{k+1},\dots, s_n)$ is given, $\binom{n-m}{s_{k+1},\dots ,s_{n}}$ is the number of ways to label the down steps of
  this Dyck prefix with labels from $\{0,1,\dots,n-m-1\}$ such that there is exactly one down step labeled $j$ for each
  $0\le j\le n-m-1$ and the labels of consecutive down steps are in decreasing order. Thus
    \[
    B_{n,k,m} =|\Dyck_{n-k,m-k}(n-m-1;1^{n-m})|.
  \]
By Lemma~\ref{lem:dyck_n,i} we obtain the formula.
\end{proof}

By \eqref{eq:PSA} and Lemma~\ref{lem:B}, we obtain the following proposition.
\begin{proposition}\label{prop:vPS}
We have
\begin{multline*}
\vol\FF_{\PS_{n+2}}(a_1,\dots ,a_k,b^{n-k+1}) \\
=\sum_{m=0}^{n} \binom nm b^{n-m} (m-k+1)(n-k+1)^{n-m-1} A_{k,m}(a_1,\dots,a_k),
\end{multline*}
where
\[
A_{k,m}(a_1,\dots,a_k) = \sum_{\substack{s_1+\dots+s_{k}=m\\ (s_1,\dots,s_{k})\ge (1^k)}}
\binom{m}{s_{1},s_{2},\dots ,s_{k}}a_{1}^{s_1}\dots a_{k}^{s_{k}}.
\]
\end{proposition}

By Proposition~\ref{prop:vPS}, in order to compute $\vol\FF_{\PS_{n+2}}(a_1,\dots ,a_k,b^{n-k+1})$, it is enough to find
$A_{k,m}(a_1,\dots,a_k)$.  For $k=0,1$, using this method we can easily recover the following formulas in \cite{Benedetti2019,
  PitmanStanley}:
\begin{align*}
\vol\FF_{\PS_{n+2}}(a^{n+1}) & =a^{n}(n+1)^{n-1},\\
\vol\FF_{\PS_{n+2}}(a,b^{n}) & =a(a+nb)^{n-1}.
\end{align*}

We now find a formula for this volume for $\vol\FF_{\PS_{n+2}}(a_1,\dots ,a_k,b^{n-k+1})$ for $k=2,3$. 

\begin{proposition}
For positive integers $a$, $b$, and $c$, we have
\[
\vol\FF_{\PS_{n+2}}(a,b,c^{n-1})=(a+b-c)(a+b+(n-1)c)^{n-1}-(b-c)(b+(n-1)c)^{n-1}.
\]
\end{proposition}
\begin{proof}
By Proposition~\ref{prop:vPS},
\[
  \vol\FF_{\PS_{n+2}}(a,b,c^{n-1}) =
  \sum_{m=0}^n \binom{n}{m}c^{n-m}(m-1)(n-1)^{n-m-1} A_{2,m}(a,b),
\]
where $A_{2,0}(a,b) =A_{2,1}(a,b)=0$  and for $m>2$,
\[
    A_{2,m}(a,b)=\sum_{\substack{i+j=m\\ (i,j)\ge (1,1)}}\binom{m}{i,j}a^{i}b^j
    =(a+b)^m-b^m.
\]
Thus
\begin{align}
\notag\vol\FF_{\PS_{n+2}}(a,b,c^{n-1}) 
&= \sum_{m=2}^n \binom{n}{m}c^{n-m}\left((a+b)^m-b^m\right)(m-1)(n-1)^{n-m-1}\\
\label{eq:vol(a,c,b,)_1} &=\dfrac{1}{n-1}\left(g_n(a+b,c^{n-1})-g_n(b,c^{n-1})-f_n(a+b,c^{n-1})+f_n(b,c^{n-1})\right),
\end{align}
where
\begin{align*}
f_n(x,y)&=\sum_{m=0}^n \binom{n}{m}x^my^{n-m} = (x+y)^n,\\  
g_n(x,y)&=\sum_{m=0}^nm \binom{n}{m}x^my^{n-m} = nx(x+y)^{n-1}.
\end{align*}
Simplifying \eqref{eq:vol(a,c,b,)_1} we obtain the result.
\end{proof}

In a similar way one can check $A_{3,m}(a,b,c)=(a+b+c)^m-(b+c)^m-ac^{m-1}$ and obtain the following proposition. We omit the details.

\begin{proposition}
For positive integers $a$, $b$, $c$, and $d$, we have
\begin{multline*}
    \vol\FF_{\PS_{n+2}}(a,b,c,d^{n-2})=(a+b+c-2d)(a+b+c+(n-2)d)^{n-1}\\
    -(b+c-2d)(b+c+(n-2)d)^{n-1}-na(c-d)(c+(n-2)d)^{n-2}.
\end{multline*}
\end{proposition}

\subsection{Volumes of flow polytopes for the Caracol graph.}
In \cite{Benedetti2019}, Benedetti et al. computed $\vol\FF_{\Car_{n+1}}(a^n)$
and $\vol\FF_{\Car_{n+1}}(a,b^{n-1})$ using unified diagrams and conjectured a
formula for $\vol\FF_{\Car_{n+1}}(a,b,c^{n-2})$, see
Proposition~\ref{prop:VolCar(a,b,c)} below. In this subsection we prove their
conjecture.  As before, for simplicity, we consider $\Car_{n+2}$ instead of
$\Car_{n+1}$.

The Caracol graph $\Car_{n+2}$ has $3n-1$ edges and
$\vec{t}':=(\outdeg(1)-1,\dots,\outdeg(n+1)-1)=(n-1,1,1,\dots,1,0)$. Note that
$\vec{s}=(s_1,\dots,s_{n+1}) \geq \vec{t}'$ implies $s_{n+1}=0$. Thus, by Lidskii's formula,
\begin{multline*}
\vol\FF_{\Car_{n+2}}(a_1,\dots  ,a_{n+1}) \\
= \sum_{\substack{s_1+\dots+s_{n}=2n-2\\(s_1,\dots,s_{n})\geq(n-1,1^{n-1})}}
  \binom{2n-2}{s_{1},\dots ,s_{n}}a_{1}^{s_1}\dots a_{n}^{s_{n}}K_{\Car_{n+2}|_{n+1}}((s_1,\dots,s_{n})-(n-1,1^{n-1})) .
\end{multline*}

Our goal is to find a formula for $X:=\vol\FF_{\Car_{n+2}}(a,b,c^{n-1})$. By the
above equation,
\begin{multline*}
X= \sum_{\substack{s_1+\dots+s_{n}=2n-2\\ (s_1,\dots,s_{n})\ge (n-1,1^{n-1})}}
\binom{2n-2}{s_{1},s_{2},\dots ,s_{n}}a^{s_1}b^{s_{2}} c^{s_{3}+\dots+s_{n}}\\
\times K_{\Car_{n+2}|_{n+1}}(s_1-n+1,s_2-1,\dots,s_n-1).
\end{multline*}
By replacing $s_1$ by $s_1+n-2$, we obtain
\begin{multline*}
X=\sum_{\substack{s_1+\dots+s_{n}=n\\ (s_1,\dots,s_{n})\ge (1^{n})}}
\binom{2n-2}{s_{1}+n-2,s_{2},\dots ,s_{n}}a^{s_1+n-2}b^{s_{2}} c^{s_{3}+\dots+s_{n}}\\
\times K_{\Car_{n+2}|_{n+1}}(s_1-1,\dots,s_n-1).
\end{multline*}
Considering $p=s_1$, $q=s_2$, and $r=s_3+\dots+s_n$ separately, we can rewrite the above equation as
\begin{equation}
  \label{eq:8}
X =  \sum_{\substack{p+q+r=n\\ (p,q)\ge (1,1)}}
\binom{2n-2}{p+n-2,q,r}a^{p+n-2}b^q c^r A(p,q,r),
\end{equation}
where
\[
A(p,q,r)=\sum_{\substack{s_3+\dots+s_{n}=r\\(p,q,s_3,\dots,s_{n})\geq (1^{n})}}
\binom{r}{s_3,\dots,s_{n}} K_{\Car_{n+2}|_{n+1}}(p-1,q-1,s_3-1,\dots,s_n-1).
\]

In the next two lemmas we find a formula for $A(p,q,r)$ using labeled Dyck paths.

Note that every Dyck path of length $2n$ can be expressed uniquely as a sequence
$UD^{d_n}UD^{d_{n-1}}\dots UD^{d_1}$ of up steps $U$ and down steps $D$ for some
$n$-tuple $(d_1,\dots,d_n)\in\ZZ_{\ge0}^n$ such that $d_1+\dots+d_n=n$ and
$(d_1,\dots,d_n)\ge (1^n)$.  For nonnegative integers $a_1,\dots,a_n$ whose sum
is at most $n$, let
\[
D_n(a_1,\dots,a_n) := \{UD^{d_n}UD^{d_{n-1}}\dots UD^{d_1} \in \Dyck_n: d_i\ge a_i\}.
\]

\begin{lemma}
\label{lem:Car1}
Let $(s_1,\dots,s_n)\in \ZZ_{\geq 0}^n$ with $\sum_{i=1}^n s_i=n$ and
$(s_1,\dots,s_n)\geq (1^n)$.  Then
\[
  K_{\Car_{n+2}|_{n+1}}(s_1-1,\dots,s_{n}-1) = |D_{n-1}(s_2,\dots,s_n)|.
\]
\end{lemma}
\begin{proof}
Note that $\Car_{n+2}|_{n+1}$ is a directed graph on $[n+1]$ with edges
$(1,i)$ for $2\leq i \leq n+1$ and $(j,j+1)$ for $2\leq j \leq n$. By
definition of Kostant partition function,
$K_{\Car_{n+2}|_{n+1}}((s_1,\dots,s_{n})-(1^n))$ is the number of nonnegative
integer solutions $\{b_{1,i}, b_{j,j+1}: 2\leq i \leq n+1, 2\leq j \leq n \}$ satisfying 
\begin{align*}
    b_{1,2}+b_{1,3}+\dots+b_{1,n+1}&=s_1-1,
    \\b_{2,3}-b_{1,2}&=s_2-1,
    \\b_{j,j+1}-b_{j-1,j}-b_{1,j}&=s_{j}-1, \qquad (3\le j\le n)
    \\-b_{n,n+1}-b_{1,n+1}&=-(s_1+\dots+s_n)+n.
\end{align*}
Since $\sum_{i=1}s_i = n$, we must have $b_{1,n+1}=b_{n,n+1}=0$ and
 the above equations are equivalent to
\begin{align*}
 b_{1,2}+b_{1,3}+\dots+b_{1,n}&=s_1-1,\\
  b_{j,j+1}&=(s_2+\dots+s_{j})+(b_{1,2}+\dots+b_{1,j})-(j-1), \qquad (2\le j\le n).
\end{align*}
Thus the integers $b_{j,j+1}$ for $2\le j\le n$ are completely determined by the
integers $b_{1,i}$ for $2\le i\le n+1$. Moreover, the condition $b_{j,j+1}\ge0$
for $2\le j\le n$ is equivalent to $(s_2,\dots,s_{n})+(b_{12},\dots,b_{1n})\geq (1^{n-1})$
in dominance order. Hence $K_{\Car_{n+2}|_{n+1}}((s_1,\dots,s_{n})-(1^n))$ is
the number of $(n-1)$-tuples $(b_{12},b_{13},\dots,b_{1n})\in \ZZ_{\geq
  0}^{n-1}$ such that $b_{12}+b_{13}+\dots+b_{1n}=s_1-1$ and
$(s_2,\dots,s_{n})+(b_{12},\dots,b_{1n})\geq (1^{n-1})$.

Now let $d_i=s_{i+1}+b_{1,i+1}$ for $1\le i\le n-1$. Then we can reinterprete
$K_{\Car_{n+2}|_{n+1}}((s_1,\dots,s_{n})-(1^n))$ as the number of $(n-1)$-tuples
$(d_1,\dots,d_{n-1})\in \ZZ_{\geq 0}^{n-1}$ such that $d_1+\dots+d_{n-1}=n-1$,
$(d_1,\dots,d_{n-1})\geq (1^{n-1})$ and $d_i\ge s_{i+1}$ for $1\le i\le
n-1$. Since the condition $(d_1,\dots,d_{n-1})\geq (1^{n-1})$ is equivalent to
the condition $UD^{d_{n-1}} UD^{d_{n-2}}\cdots UD^{d_{1}}\in\Dyck_{n-1}$, we
obtain the desired result.
\end{proof}

\begin{lemma}
Let $p,q$ and $r$ be fixed nonnegative integers with $p+q+r=n$ and $(p,q)\geq (1,1)$.
Then
\label{lem:Car2}
\[
A(p,q,r)   =(p+q-1)\binom{n+p-2}{n-1}(n-1)^{r-1}-\binom{n+p-2}{n}(n-1)^r.
\]
\end{lemma}
\begin{proof}
  By Lemma~\ref{lem:Car1},
  \[
A(p,q,r)= \sum_{\substack{s_3+\dots+s_{n}=r\\(p,q,s_3,\dots,s_{n})\geq (1^{n})}} 
\binom{r}{s_3,\dots,s_{n}} |D_{n-1}(q,s_3,\dots,s_n)|.
\]

We will give a combinatorial interpretation of each summand in the above
formula using labeled Dyck paths.  Let $s_3,\dots,s_n$ be nonnegative integers
satisfying $s_3+\dots+s_{n}=r$ and $(p,q,s_3,\dots,s_{n})\geq (1^{n})$.
Consider a Dyck path
$\pi=UD^{d_{n-1}}UD^{d_{n-2}}\dots UD^{d_{1}}\in D_{n-1}(q,s_3,\dots,s_n)$.
Then $d_1\ge q$ and $d_i\ge s_{i+1}$ for $2\le i\le n-1$. Now we label the down
steps of $\pi$ except the last consecutive down steps $D^{d_1}$ as follows:
\begin{itemize}
\item  Distribute the $r$ labels $1,2,\dots,r$, each label occurring
exactly once, to the sequences $D^{d_{n-1}}, D^{d_{n-2}},\dots, D^{d_2}$
consecutive down steps of $\pi$ so that the sequence $D^{d_i}$ gets $s_{i+1}$
labels. There are $\binom{r}{s_3,\dots,s_{n}}$ ways to do this.
\item Add $d_i-s_{i+1}$ zero labels to the sequence $D^{d_i}$ and
arrange the labels in weakly decreasing order. 
\end{itemize}
Considering the resulting objects of this process we obtain that
$\binom{r}{s_3,\dots,s_{n}} |D_{n-1}(q,s_3,\dots,s_n)|$ is the number of Dyck
paths $\pi=UD^{d_{n-1}}UD^{d_{n-2}}\dots UD^{d_{1}}$ together with a labeling on
the down steps except the last consecutive down steps $D^{d_{1}}$ satisfying the
following conditions:
\begin{enumerate}
\item  $d_i\ge s_{i+1}$ for $2\le i\le n-1$.
\item  $q\le d_1\le n-1-r$.
\item The number of down steps labeled $i$ is $1$ for $1\le i\le r$.
\item The number of down steps labeled $0$ is $n-1-r-d_1$.
\item The labels of any consecutive down steps are weakly decreasing.
\end{enumerate}
Summing over all possible $s_3,\dots,s_n$ we obtain that $A(p,q,r)$ is the
number of Dyck paths $\pi=UD^{d_{n-1}}UD^{d_{n-2}}\dots UD^{d_{1}}$ together
with a labeling on the down steps of its prefix
$UD^{d_{n-1}}UD^{d_{n-2}}\dots UD^{d_{2}}$ from $(0,0)$ to $(2n-3-d_1,d_1-1)$
satisfying the above conditions except (1). This implies 
\[
A(p,q,r)= \sum_{d_1=q}^{n-1-r} |\Dyck_{n-2,d_1-1}(r;n-1-r-d_1,1^{r})|.
\]
By Lemma~\ref{lem:dyck_n,i},
\begin{align*}
  A(p,q,r) & = \sum_{d_1=q}^{n-1-r} \frac{d_1}{n-1} \Multiset{n-1}{n-1-r-d_1} (n-1)^r\\
&= (n-1)^{r-1}\sum_{d_1=q}^{n-1-r} d_1 \binom{2n-3-r-d_1}{n-2}.
\end{align*}
Replacing $d_1$ by $n-1-r-i$, we have
\[
  A(p,q,r) = (n-1)^{r-1} \sum_{i=0}^{p-1} (n-1-r-i) \binom{n-2+i}{n-2}.
\]
Since
\begin{align*}
(n-1-r-i) \binom{n-2+i}{n-2} &= \left( 
(2n-2-r) - (n-1+i) \right) \binom{n-2+i}{n-2}\\
&= (2n-2-r) \binom{n-2+i}{n-2} - (n-1) \binom{n-1+i}{n-1}\\
&= (n-1-r) \binom{n-2+i}{n-2} - (n-1) \binom{n-2+i}{n-1},
\end{align*}
we have
\[
  A(p,q,r) = (n-1)^{r-1} \left((p+q-1) \sum_{i=0}^{p-1} \binom{n-2+i}{n-2}
-(n-1)\sum_{i=0}^{p-2}\binom{n-1+i}{n-1}    \right).
\]
Finally the identity $\sum_{i=0}^{k}\binom{m+i}{m}=\binom{m+k+1}{m+1}$
finishes the proof.
\end{proof}

Now we are ready to compute $X=\vol\FF_{\Car_{n+2}}(a,b,c^{n-1})$.

\begin{proposition}\cite[Conjecture~6.16]{Benedetti2019}
\label{prop:VolCar(a,b,c)}
For positive integers $a$, $b$, and $c$, we have
\[
\vol\FF_{\Car_{n+2}}(a,b,c^{n-1})=C_{n-1}a^{n-1}(a+nb)(a+b+(n-1)c)^{n-2}.
\]
\end{proposition}
\begin{proof}
  By \eqref{eq:8} and Lemma~\ref{lem:Car2}, we have
  \[
  \vol\FF_{\Car_{n+2}}(a,b,c^{n-1})  = X = Y-Z,
\]
where
\[
Y =  \sum_{\substack{p+q+r=n\\ (p,q)\ge (1,1)}}
  \binom{2n-2}{p+n-2,q,r}a^{p+n-2}b^q c^r (p+q-1)\binom{n+p-2}{n-1}(n-1)^{r-1},
\]
\[
Z =  \sum_{\substack{p+q+r=n\\ (p,q)\ge (1,1)}}
\binom{2n-2}{p+n-2,q,r}a^{p+n-2}b^q c^r \binom{n+p-2}{n}(n-1)^r.
\]
Note that in the above two sums, the condition $(p,q)\ge(1,1)$ can be omitted
since the summand is zero if $p=0$ or $(p,q)=(1,0)$.
Thus
\begin{align*}
Y    &= \frac{a^{n-1}}{n-1}\binom{2n-2}{n-1} \sum_{p+q+r=n}
(p+q-1) \binom{n-1}{p-1,q,r}a^{p-1}b^q (c(n-1))^r,\\
Z &=  \sum_{p+q+r=n}
\binom{2n-2}{p+n-2,q,r}a^{p+n-2}b^q c^r \binom{n+p-2}{n}(n-1)^r\\
&=a^{n}\binom{2n-2}{n}  \sum_{p+q+r=n}
\binom{n-2}{p-2,q,r}a^{p-2}b^q (c(n-1))^r .
\end{align*}
Using  the multinomial theorem
\[
\sum_{i+j+k=m}  \binom{m}{i,j,k}x^iy^j z^k t^{i+j} = (xt+yt+z)^m,
\]
and its derivative with respect to $t$, i.e,
\[
\sum_{i+j+k=m} (i+j) \binom{m}{i,j,k} x^i y^j z^k t^{i+j-1} = m(x+y)(xt+yt+z)^{m-1},
\]
we obtain
\begin{align*}
Y &= C_{n-1}a^{n-1}n(a+b)(a+b+(n-1)c)^{n-2},\\
Z &=C_{n-1}a^n(n-1)(a+b+(n-1)c)^{n-2},
\end{align*}
and the proof follows.
\end{proof}

\section*{Acknowledgments}
The authors would like to thank Alejandro Morales for informing them  
that Theorems~\ref{thm:PS} and \ref{thm:main} are equivalent to
Theorems~\ref{Thm:PS} and \ref{Thm:Car}. They also thank Nathan Williams for
helpful discussion.

\end{document}